\documentclass[12pt,letter]{amsart}
\usepackage[all,cmtip]{xy}
\usepackage{enumerate, comment}
\usepackage{mathrsfs}
\usepackage{ amssymb, latexsym, amsmath}
\usepackage{fullpage}
\usepackage{url}
\usepackage{hyperref}
\usepackage{helvet}
\usepackage{amsfonts}
\usepackage{tikz}
\usepackage{amssymb}
\usepackage{amsthm}
\usepackage{amsmath}
\usepackage{amscd}
\usepackage{stmaryrd}
\usepackage[mathscr]{eucal}
\usepackage{indentfirst}
\usepackage{graphicx}
\usepackage{tikz-cd}
\usepackage{graphics}
\usepackage{pict2e}
\usepackage{epic}
\usepackage[normalem]{ulem}
\usepackage[OT2,T1]{fontenc}
\numberwithin{equation}{section}
\usepackage[margin=3.4cm]{geometry}
 
\usepackage{color}
\DeclareSymbolFont{cyrletters}{OT2}{wncyr}{m}{n}
\DeclareMathSymbol{\Sha}{\mathalpha}{cyrletters}{"58}

\newcommand{\Selg}{\operatorname{Sel}_{p^{\infty}}(g/\Q_{\infty})}
\newcommand{\Selpm}{\operatorname{Sel}_{p^{\infty}}^{\pm}(g/\Q_{\infty})}
\newcommand{\Z}{\mathbb{Z}}
\newcommand{\p}{\mathfrak{p}}
\newcommand{\Q}{\mathbb{Q}}
\newcommand{\g}{\operatorname{Ad}^0\bar{\rho}}
\newcommand{\op}[1]{\operatorname{#1}}
\newcommand{\F}{\mathbb{F}}
\newcommand\mtx[4] { \left( {\begin{array}{cc}
   #1 & #2 \\
   #3 & #4 \\
  \end{array} } \right)}

\theoremstyle{plain}
 \theoremstyle{definition}
\newtheorem{Th}{Theorem}[section]
\newtheorem{Lemma}[Th]{Lemma}
\newtheorem{convention}[Th]{Convention}
\newtheorem{hypothesis}[Th]{Hypothesis}
 
\newtheorem{Corollary}[Th]{Corollary}
\newtheorem{Proposition}[Th]{Proposition}
\newtheorem{Remark}[Th]{Remark}
 \theoremstyle{definition}
\newtheorem{Definition}[Th]{Definition}
\newtheorem{Conjecture}[Th]{Conjecture}

\begin{document}

\title{Constructing Galois representations with large Iwasawa {\large$\lambda$}-Invariant}
\author{Anwesh Ray}
\address{Anwesh Ray\newline Centre de recherches mathématiques,
Université de Montréal,
Pavillon André-Aisenstadt,
2920 Chemin de la tour,
Montréal (Québec) H3T 1J4, Canada}

\begin{abstract}
Let $p\geq 5$ be a prime. We construct modular Galois representations for which the $\Z_p$-corank of the $p$-primary Selmer group (i.e., $\lambda$-invariant) over the cyclotomic $\Z_p$-extension is large. More precisely, for any natural number $n$, one constructs a modular Galois representation such that the associated $\lambda$-invariant is $\geq n$. The method is based on the study of congruences between modular forms, and leverages results of Greenberg and Vatsal. Given a modular form $f_1$ satisfying suitable conditions, one constructs a congruent modular form $f_2$ for which the $\lambda$-invariant of the Selmer group is large. A key ingredient in acheiving this is the Galois theoretic lifting result of Fakruddin-Khare-Patrikis, which extends previous work of Ramakrishna. The results are subject to certain additional hypotheses, and are illustrated by explicit examples. 

\medskip

\noindent\textsc{R\'esum\'e.} Soit $p\geq 5$ un nombre premier. Nous construisons des repr\'esentations galoisiennes modulaires pour lesquelles le $\Z_p$-corank du groupe de Selmer $p$-primary (i.e., $\lambda$-invariant) sur l'extension cyclotomique $\Z_p$-est grand. Plus pr\'ecis\'ement, pour tout entier naturel $n$, on construit une repr\'esentation galoisienne modulaire telle que l'invariant $\lambda$ associ\'e soit $\geq n$. La m\'ethode est bas\'ee sur l'\'etude des congruences entre les formes modulaires, et s'appuie sur les r\'esultats de Greenberg et Vatsal. \'Etant donn\'e une forme modulaire $f_1$ satisfaisant des conditions convenables, on construit une forme modulaire congruente $f_2$ pour laquelle le $\lambda$-invariant du groupe de Selmer est grand. Un ingr\'edient cl\'e pour y parvenir est le résultat de levage th\'eorique de Galois de Fakruddin-Khare-Patrikis, qui prolonge les travaux ant\'erieurs de Ramakrishna. Les r\'esultats sont soumis \`a certaines hypothèses suppl\'ementaires, et sont illustr\'es par des exemples explicites.

\end{abstract}

\maketitle
\section{Introduction}
\par  The Iwasawa theory of elliptic curves was initiated by B.Mazur in \cite{mazur} who studied the growth of the rank of an elliptic curve up certain infinite towers of number fields. It was proven by K.~Kato \cite{kato} and D.~Rohlrich \cite{rohlrich} that the Mordell--Weil rank of a rational elliptic curve in the number fields contained in the cyclotomic $\Z_p$-extension of $\Q$ is bounded. The bound is given in terms of a certain Iwasawa invariant, known as the $\lambda$-invariant. The behaviour and properties of such Iwasawa invariants has been the subject of much conjecture and contemplation. R.~Greenberg conjectured that the Iwasawa $\mu$-invariant must vanish when the residual representation is irreducible and proved in \cite{iecgreenberg} that the Iwasawa $\lambda$-invariant of a rational elliptic curve may be arbitrarily large, provided the $\mu=0$ conjecture holds. This result may be viewed as an Iwasawa theoretic analogue of the \textit{rank boundedness conjecture} for elliptic curves (see \cite{parketal} for a survey), and the proof crucially relies on the arithmetic geometry of elliptic curves. More recently, there has been growing interest in the Iwasawa theory of modular forms and their associated Galois representations. Such representations include (but are not limited to) those arising from elliptic curves. Iwasawa invariants for modules arising from modular Galois representations provide key insights into the arithmetic of such objects. Let $f$ be a Hecke eigencuspform and $p$ a prime number. The Selmer group associated to $f$ is defined over the cyclotomic $\Z_p$-extension $\Q_{\infty}$ of $\Q$, and the algebraic structure of this Selmer group is of significant interest in Iwasawa theory. The Iwasawa $\lambda$-invariant is the $\Z_p$-rank of its Pontryagin dual.
\par In this manuscript, the aforementioned result of Greenberg is generalized. The constructions apply not only to elliptic curves, but modular forms and their associated Galois representations. A Hecke eigencuspform that coincides with a rational elliptic curve must be of weight $2$, with rational Fourier coefficients. The geometric tools at one's disposal when working with elliptic curves are no longer available in the general setting. Second, the results of Greenberg apply to $p$-ordinary elliptic curves, while the constructions in this paper apply to both $p$-ordinary and $p$-supersingular elliptic curves. In the supersingular case, the Selmer groups are replaced by the natural analogs, known as signed Selmer groups. These were first introduced by S.~Kobayashi \cite{kobayashi} for elliptic curves, and later generalized to modular forms by A.~Lei in \cite{Lei}. For the prime $p=3$, B.~D.~Kim in \cite{Kim} studied a related question for Kobayashi's plus and minus Selmer groups associated to $p$-supersingular elliptic curves. Such results were proved by analyzing the variation of $\lambda$-invariants in families of elliptic curves constructed by K.~Rubin and A.~Silverberg (see \cite{rubinsilver}). These elliptic curves have constant mod-$p$ Galois representation.
\par The new technique we introduce in order to study this problem in Iwasawa theory is based on Galois deformation theory, and has its origin in the Serre's conjecture. The technique was initially pioneered by R. Ramakrishna, who proved a geometric lifting theorem in \cite{ravi1, ravi2}, which provided crucial evidence for Serre's conjecture. In this paper, we use Galois deformation theory to show that certain residual representations
\[\bar{\rho}:\op{Gal}(\bar{\Q}/\Q)\rightarrow \op{GL}_2(\bar{\F}_p)\]lift to modular Galois representations with large $\lambda$-invariant (see Theorem \ref{th52}). There are a number of hypotheses on the residual representation. We crucially use the results of R.~Taylor \cite{taylor} and M.~Kisin \cite{kisin} on the Fontaine-Mazur conjecture, and our assumptions on $\bar{\rho}$ must account for an exceptional case in which the Fontaine-Mazur conjecture is not proved. We use the result of Fakhruddin-Khare-Patrikis \cite[Theorem A]{FKP} which gives some control on the local representations at a set of primes away from $\{p\}$. The reader is referred to Theorems \ref{techtheorem} and \ref{th52} for precise statements. Results are proved not only for $p$-ordinary Galois representations but also for $p$-crystalline non-ordinary Galois representations as well, and they apply to all primes $p\geq 5$. In the special case when the prime is non-ordinary, the weight of the modular form is $2$, and the $p$-th Fourier coefficient is non-zero, we need to additionally assume that the relevant signed dual Selmer groups are torsion over the Iwasawa algebra. The Theorem \ref{lastresult} constructs an explicit example of a compatible system of Galois representations to which the main results apply.
\par The paper consists of 5 sections. Preliminary notions are introduced in section $\ref{s2}$ and we define the various Selmer groups considered in this paper. In section $\ref{s3}$, results of Greenberg-Vatsal \cite{greenbergvatsal}, B.D.Kim \cite{Kim} and Hatley-Lei \cite{hatleylei} on congruent Galois representations are discussed. In section $\ref{s4}$, a residual representation $\bar{\rho}$ is lifted to a characteristic zero modular Galois representation so that it may satisfy favorable conditions. These conditions imply that the $\lambda$-invariant of the associated Selmer group is large. In section $\ref{s5}$, the main results of this paper are proved.

\section{Preliminary notions}\label{s2}
\par In this section, we introduce notation and discuss some preliminary notions. First, we begin with some standard notation. At each prime number $v$, we fix an embedding $\iota_v:\bar{\Q}\hookrightarrow\bar{\Q}_v$. Denote by $\op{G}_v$ the Galois group $\op{Gal}(\bar{\Q}_v/\Q_v)$ and note that the embedding $\iota_v$ gives rise to an inclusion of $\op{G}_v$ into the absolute Galois group $\op{Gal}(\bar{\Q}/\Q)$. Fix an odd prime number $p$ and a normalized Hecke eigencuspform $g$ of weight $k\geq 2$ on $\Gamma_1(N)$. Associated to $g$ is the Galois representation \[\rho_g:\op{Gal}(\bar{\Q}/\Q)\rightarrow \op{GL}_2(\bar{\Q}_p),\] see \cite{deligne} for details. Since $\op{Gal}(\bar{\Q}/\Q)$ is compact, the image of $\rho_g$ is contained in $\op{GL}_2(K)$, for a finite extension $K$ of $\Q_p$. We set $\mathcal{O}$ to be the valuation ring of $K$ and let $\F$ be its residue field. Denote by $\op{V}_g$ the underlying $K$-vector space on which the Galois group $\op{Gal}(\bar{\Q}/\Q)$ acts via $\rho_g$ and choose a Galois stable $\mathcal{O}$-lattice $\op{T}_g$ contained in $\op{V}_g$. We shall simply denote the integral representation on $\op{T}_g$ by
\[\rho_g:\op{Gal}(\bar{\Q}/\Q)\rightarrow \op{GL}_2(\mathcal{O})\] for ease of notation. Let $\op{W}_g$ be the $p$-divisible Galois module $\op{V}_g/\op{T}_g$. At each prime $v$, the restriction of $\rho_g$ to $\op{G}_v$ is denoted $\rho_{g|v}$. Recall that $g$ is $p$-\textit{ordinary} if its $p$-th Fourier coefficient is a $p$-adic unit. We shall require the following simplifying hypothesis.\begin{hypothesis}
We say that $g$ satisfies $(\star)$ at $p$ if at least one of the following conditions is satisfied:
\begin{enumerate}
    \item $g$ is $p$-ordinary,
    \item $p\nmid N$ and $p\geq k$.
\end{enumerate}

\end{hypothesis}

Assume that $g$ satisfies $(\star)$. Let $I_p$ denote the inertia subgroup of $\op{G}_p$. Set $\chi:\op{G}_{\Q, \{p\}}\rightarrow \op{GL}_1(\Z_p)$ to denote the cyclotomic character. Then, if $g$ is $p$-ordinary, we have that
\[\rho_{g|I_p}=\mtx{\chi^{k-1}}{\ast}{0}{1}.\] If $g$ is not $p$-ordinary, we have assumed that $p\nmid N$. It is well known that in this case, the local representation $\rho_{g|p}$ is crystalline.
\par We define Selmer groups over the \textit{cyclotomic $\Z_p$-extension} of $\Q$. For $n\geq 0$, let $\Q_n$ be the subfield of $\Q(\mu_{p^{n+1}})$ degree $p^n$ over $\Q$. Note that $\Q_n$ is contained in $\Q_{n+1}$. Let $\Q_{\infty}$ be the union \[\Q_{\infty}:=\bigcup_{n\geq 0} \Q_n\]and set $\Gamma:=\op{Gal}(\Q_{\infty}/\Q)$. Note that there are isomorphisms of topological groups \[\op{Gal}(\Q_{\infty}/\Q)\xrightarrow{\sim} \varprojlim_n\op{Gal}(\Q_{n}/\Q)\xrightarrow{\sim} \Z_p.\] The extension $\Q_{\infty}$ is the cyclotomic $\Z_p$-extension of $\Q$ and $\Q_n$ is its \textit{$n$-th layer}. Choose a topological generator $\gamma\in \Gamma$ and fix an isomorphism $\Z_p\xrightarrow{\sim} \Gamma$ sending $a$ to $\gamma^a$. The Iwasawa algebra $\Lambda$ is defined as the following inverse limit
\[\Lambda:=\varprojlim_n \Z_p[\op{Gal}(\Q_n/\Q)].\] Fix an isomorphism of $\Lambda$ with the ring of formal power series $\Z_p\llbracket T\rrbracket$, by identifying $\gamma-1$ with $T$.
\par Fix a finite set of primes $S$ containing the primes dividing $Np$. Note that $\rho_g$ is unramified at all primes $v\notin S$. Set $\Q_S$ to denote the maximal algebraic extension of $\Q$ which is unramified at all primes outside $S$, and let $\op{G}_{\Q,S}$ be the Galois group $\op{Gal}(\Q_S/\Q)$.
\par First, consider the case when $g$ is ordinary at $p$. For a number field $L$ contained in $\Q_{\infty}$, set 
\[\mathcal{H}_p(g/L):=\bigoplus_{\eta|p} \frac{H^1(L_{\eta}, \op{W}_g)}{H^1_f(L_{\eta}, \op{W}_g)},\] where $H^1_f(L_{\eta}, \op{W}_g)$ is the \textit{Bloch-Kato} condition introduced in \cite{blochkato}. The reader may also refer to \cite[Definition 2.1]{ochiai}. Here, the set of primes $\eta$ range over the primes $\eta|p$ of $L$ above $p$. Denote by $\mathcal{H}_p(g/\Q_{\infty})$ the direct limit
\[\mathcal{H}_p(g/\Q_{\infty}):=\varinjlim_n \mathcal{H}_p(g/\Q_{n}).\]Following \cite[section 3]{greenbergvatsal}, for $v\neq p$, let 
\[\mathcal{H}_p(g/\Q_{\infty}):=\bigoplus_{\eta|v}\op{im}\left\{H^1(\Q_{\infty, \eta}, \op{W}_g)\rightarrow H^1(\op{I}_{\eta}, \op{W}_g) \right\}.\] In the above formula, $\eta$ runs through all primes of $\Q_{\infty}$ dividing $v$ and $\op{I}_{\eta}$ is the inertia group of $\op{Gal}(\bar{\Q}_{\infty, \eta}/\Q_{\infty, \eta})$. The $p$-primary Selmer group is the kernel of the restriction map
\[\Selg:=\op{ker}\left\{H^1(\Q_S/\Q_{\infty},\op{W}_g)\rightarrow \bigoplus_{v\in S} \mathcal{H}_v(g/\Q_{\infty})\right\}.\] Its Pontryagin dual 
\[\Selg^{\vee}:=\op{Hom}\left(\Selg, \Q_p/\Z_p\right)\] is finitely generated $\Lambda$-module. Furthermore, Kato \cite{kato} proved that the dual Selmer group $\Selg^{\vee}$ is a torsion $\Lambda$-module.
\par Next, consider the case when $g$ is not $p$-ordinary. In this setting, $\Selg^{\vee}$ is not torsion, see \cite[Proposition 6.3]{Lei} for details. Instead of working with Selmer groups that are not cotorsion, we work with \textit{signed Selmer groups} $\op{Sel}_{p^{\infty}}^{+} (g/\Q_{\infty})$ and $\op{Sel}_{p^{\infty}}^{-} (g/\Q_{\infty})$. These Selmer groups were introduced by Kobayashi in \cite{kobayashi} for $p$-supersingular elliptic curves, and their definition has been generalized for modular Galois representations by Lei in \cite{Lei}. Assume that $g$ is not $p$-ordinary. For a number field $L$ contained in $\Q_{\infty}$ set
\[\mathcal{H}_v^{\pm}(g/L):=\bigoplus_{\eta|v} \frac{H^1(L_{\eta}, \op{W}_g)}{H^1_f(L_{\eta}, \op{W}_g)^{\pm}},\] where $H^1_f(L_{\eta}, \op{W}_g)^{\pm}$ is defined in \cite{Lei}. The signed Selmer groups are defined as follows
\[\Selpm:=\op{ker}\left\{H^1(\Q_S/\Q_{\infty},\op{W}_g)\rightarrow \mathcal{H}_p^{\pm}(g/\Q_{\infty})\oplus \left(\bigoplus_{v\in S\backslash \{p\}} \mathcal{H}_v(g/\Q_{\infty})\right)\right\}.\] The Pontryagin dual 
$\Selpm^{\vee}$ is a finitely generated and $\Lambda$-module. For our results to hold, it is necessary to assume that $\Selpm^{\vee}$ is in fact torsion as a $\Lambda$-module. 
\begin{hypothesis}\label{hypothesis torsion}
Assume that the dual Selmer groups $\Selpm^{\vee}$ are torsion as $\Lambda$-modules.
\end{hypothesis}
The above hypothesis is known to be satisfied in various cases. Let $a_p(g)$ denote the Fourier coefficient of $g$ at $p$. For instance, if either $k\geq 3$ or if $a_p(g)=0$, then the above Hypothesis is known to hold, cf. \cite[Theorem 1.4]{LLZ}, \cite[Proposition 6.4]{Lei}. On the other hand, in the case when $k=2$ and $g$ is known to arise from an elliptic curve, then, the torsion property is studied by Sprung, cf. \cite[Theorem 7.14]{sprung}.
\par We introduce the Iwasawa invariants that are studied in this manuscript. Let $\rm{M}$ be a finitely generated torsion $\Lambda$-module.
By the well known structure theorem of $\Lambda$-modules, there is a map of
\[
\textrm{M}\longrightarrow  \left(\bigoplus_{i=1}^a \Lambda/(p^{\mu_i})\right)\oplus \left(\bigoplus_{j=1}^b \Lambda/(f_j(T)) \right)
\]
with finite kernel and cokernel.
Here, $\mu_i>0$ and $f_j(T)$ is a monic polynomial with non-leading coefficients divisible by $p$.
The \textit{$\mu$-invariant} and \textit{$\lambda$-invariant} of $\rm{M}$ are defined as follows
\[
\begin{split}
&\mu(\textrm{M}):=\begin{cases}0 & \textrm{ if } a=0\\
\sum_{i=1}^a \mu_i & \textrm{ if } a>0.
\end{cases}\\
&\lambda(\textrm{M}) := 
\begin{cases}0 & \textrm{ if } b=0\\
\sum_{j=1}^b \deg f_j(T) & \textrm{ if } b>0.
\end{cases}
\end{split}
\]
In order to simultaneously state our results with ease, we introduce the following convention.
\begin{convention}\label{convention}
Let $g$ be a Hecke eigencuspform satisfying $(\star)$ at $p$. Then, if $g$ is $p$-ordinary, the Selmer groups $\op{Sel}_{p^{\infty}}^+(g/\Q_{\infty})$ and $\op{Sel}_{p^{\infty}}^-(g/\Q_{\infty})$ shall both simply denote the classical Selmer group $\op{Sel}_{p^{\infty}}(g/\Q_{\infty})$. Furthermore, $\mathcal{H}_p^{\pm}(g/\Q_{\infty})$ shall denote $\mathcal{H}_p(g/\Q_{\infty})$.
\end{convention}
Assume that $g$ if is $p$-supersingular with $k=2$ and $a_p(g)\neq 0$, then, Hypothesis \ref{hypothesis torsion} is satisfied. Denote by $\mu^{\pm}(g)$ and $\lambda^{\pm}(g)$ the $\mu$ and $\lambda$-invariants of $\Selpm$ respectively. Note that in accordance with the convention above, \[\mu^{\pm}(g):=\mu(\Selg)\text{ and }\lambda^{\pm}(g):=\lambda(\Selg)\] when $g$ is $p$-ordinary. In the above formulae, the dependence on the prime "$p$" is suppressed. When $p$ is fixed, we shall use $\mu^{\pm}$ and $\lambda^{\pm}$, and we emphasize the dependence on "$p$" depending on the context by using $\mu_p^{\pm}$ and $\lambda_p^{\pm}$. The following generalizes a conjecture due to Greenberg \cite[Conjecture 1.11]{iecgreenberg} and is backed up by computational evidence.
\begin{Conjecture}
Let $g$ be a Hecke eigencuspform satisfying condition $(\star)$ and $\ddag\in \{+,-\}$ a choice of sign. If the residual representation $\bar{\rho}_g$ is irreducible, then $\mu_p^{\ddag}(g)=0$.
\end{Conjecture}
\section{Congruent Galois Representations}\label{s3}
\par Let $\tau$ be a complex number in the upper half plane and set $q:=e^{2\pi i \tau}$. Let $f_1=\sum_{n=1}^{\infty} a_n(f_1)q^n$ and $f_2=\sum_{n=1}^{\infty} a_n(f_2)q^n$ be normalized Hecke eigencuspforms and $L$ the number field generated by the Fourier coefficients of $f_1$ and $f_2$. The Fourier coefficients $a_n(f_i)$ are all contained in the ring of integers $\mathcal{O}_L$. Let $\mathfrak{p}$ be a prime ideal in $\mathcal{O}_L$ such that $\mathfrak{p}|p$.
\begin{Definition}We say that $f_1$ and $f_2$ are congruent modulo $\mathfrak{p}$ if for all but finitely many primes $\ell$,
\[a_{\ell}(f_1)\equiv a_{\ell}(f_2)\mod{\p}.\]
\end{Definition}
Let $f_1$ and $f_2$ be $\p$-congruent and $K$ denote the completion $L_{\p}$. Assume that both $f_1$ and $f_2$ satisfy $(\star)$ at $p$. Let $\mathcal{O}$ be the valuation ring of $K$, $\varpi$ the uniformizer of $\mathcal{O}$ and $\F$ the residue field $\mathcal{O}/\varpi$. Let
\[\bar{\rho}_{f_i}:\op{Gal}(\bar{\Q}/\Q)\rightarrow \op{GL}_2(\F)\] denote the mod-$\varpi$ reduction of $\rho_{f_i}$. Since $f_1$ and $f_2$ are congruent modulo $\mathfrak{p}$ it follows that their residual representations are isomorphic up to semisimplification. Denote by $N_i$ the level of $f_i$ and by $\Sigma$ the set of primes $v$ such that $v|N_1N_2p$. Let $\ddag$ be a choice of sign $+$ or $-$. Note that if $f_1$ (or equivalently, $f_2$) is $p$-ordinary, the Selmer group $\op{Sel}_{p^{\infty}}^{\ddag}(f_i/\Q_{\infty})$ is simply taken to be $\op{Sel}_{p^{\infty}}(f_i/\Q_{\infty})$.
\begin{hypothesis}\label{mainhyp}
We make the following assumptions:
\begin{enumerate}
    \item both $f_1$ and $f_2$ satisfy $(\star)$ at $p$,
    \item Hypothesis \ref{hypothesis torsion} is satisfied for both $f_1$ and $f_2$,
    \item $\bar{\rho}_{f_1}$ (or equivalently $\bar{\rho}_{f_2}$) is irreducible.
\end{enumerate}
\end{hypothesis}
\par As a result of the second assumption, it follows that $\bar{\rho}_{f_1}$ is isomorphic to $\bar{\rho}_{f_2}$.
\begin{Lemma}\label{congruencelemma}
Let $f_1$ and $f_2$ be as above. Then $f_1$ is $p$-ordinary if and only if $f_2$ is $p$-ordinary.
\end{Lemma}
\begin{proof}
As is well known, $\bar{\rho}_{f_i|I_p}$ acts by fundamental characters of level one when $f_i$ is $p$-ordinary, and by fundamental characters of level two otherwise. The reader may refer to the discussion in \cite[section 2.1]{ribetstein} for instance. The assertion of the Lemma follows.
\end{proof}
We introduce imprimitive Selmer groups which will play a crucial role in what follows. Recall Convention $\ref{convention}$ which unifies notation in what follows. For $i=1,2$, the imprimitive Selmer groups associated to $f_i$ are defined as follows
\[\op{Sel}_{p^{\infty}}^{\ddag,\Sigma}(f_i/\Q_{\infty}):=\op{ker} \left\{ H^1(\Q_{\Sigma}/\Q_{\infty}, \op{W}_{f_i})\xrightarrow{\op{res}_p} \mathcal{H}_p^{\ddag}(f_i/\Q_{\infty}) \right \}.\]
Denote by $\mu^{\ddag,\Sigma}(f_i)$ and $\lambda^{\ddag,\Sigma}(f_i)$ the $\mu$ and $\lambda$-invariants of $\op{Sel}_{p^{\infty}}^{\ddag, \Sigma}(f_i/\Q_{\infty})$ respectively. Recall that the primitive Selmer group $\op{Sel}_{p^{\infty}}^{\ddag}(f_i/\Q_{\infty})$ is the kernel of the restriction map
\[\op{res}_{\Sigma}:H^1(\Q_{\Sigma}/\Q_{\infty}, \op{W}_{f_i})\rightarrow \left(\bigoplus_{v\in \Sigma\backslash \{p\}}\mathcal{H}_v(f_i/\Q_{\infty})\right) \oplus \mathcal{H}_p^{\ddag}(f_i/\Q_{\infty}).\] Since $\op{Sel}_{p^{\infty}}(f_i/\Q_{\infty})^{\vee}$ is a torsion $\Lambda$-module, it follows that the above defining map $\op{res}_{\Sigma}$ is surjective, see \cite[Proposition 2.4]{ponsinet} or \cite[Corollary 2.9]{hatleylei} for further details. As a result, there is a short exact sequence
\begin{equation}\label{ses} 0\rightarrow \op{Sel}_{p^{\infty}}^{\ddag}(f_i/\Q_{\infty})\rightarrow \op{Sel}_{p^{\infty}}^{\ddag,\Sigma}(f_i/\Q_{\infty})\\
\rightarrow \bigoplus_{v\in \Sigma\backslash \{p\}} \mathcal{H}_v(f_i/\Q_{\infty})\rightarrow 0.\end{equation}
For a prime $v\neq p$, denote by $\mu_v(f_i)$ (resp. $\lambda_v(f_i)$) the $\mu$-invariant (resp. $\lambda$-invariant) of $\mathcal{H}_v(f_i/\Q_{\infty})^{\vee}$. Let $\op{I}_v$ be the inertia subgroup of $\op{G}_{v}$ and set $\op{V}_{f_i}':=\op{V}_{f_i}|_{\op{I}_v}$ to be the maximal quotient on which $\op{I}_v$ acts trivially. Denote by $P_v(X)$ the characteristic polynomial of $\op{Frob}_v$ acting on $\op{V}_{f_i}'$,
\[P_v(X)=\op{det}\left((1-\op{Frob}_v)|_{\op{V}_{f_i}'}\right).\] Let $\widetilde{P}_v(X)$ be the mod-$p$ reduction of $P_v(X)$ and $d_v$ the multiplicity of $v^{-1}\in \Z/p\Z$ as a root of $\widetilde{P}_v(X)$. Set $s_v$ for the largest power of $p$ such that $v^{p-1}\equiv 1\mod{ps_v}$.
\begin{Lemma}\cite[Proposition 2.4]{greenbergvatsal}
The $\mu$-invariant $\mu_v(f_i)$ is equal to $0$ and the $\lambda$-invariant $\lambda_v(f_i)$ is equal to $s_v d_v$.
\end{Lemma}

\begin{Corollary}\label{cor34}
Let $f_1$ and $f_2$ be $\p$-congruent satisfying Hypothesis $\ref{mainhyp}$. Let $v\in \Sigma\backslash \{p\}$ be such that the following assertions are satisfied
\begin{enumerate}
    \item $\bar{\rho}_{f_1}(g)=\op{Id}$ (or equivalently $\bar{\rho}_{f_2}(g)=\op{Id}$) for all $g\in \op{G}_v$.
    \item The prime $v$ divides $N_2$ and not $N_1$.
\end{enumerate}
Then, we have that
\[\lambda_v(f_1)-\lambda_v(f_2)\geq 1.\]
\end{Corollary}
\begin{proof}
Since $v$ divides $N_2$, it follows that $\dim \op{V}_{f_2}|_{\op{I}_v}\leq 1$ and hence, $\lambda_v(f_2)\leq s_v$. On the other hand, $\op{V}_{f_1}'=\op{V}_{f_1}$ and $\widetilde{P}_v(X)=(1-vX)^2$. As a result, $\lambda_v(f_1)= 2s_v$. It follows that 
\[\lambda_v(f_1)-\lambda_v(f_2)\geq s_v\geq 1.\]
\end{proof}
By the structure theory of $\Lambda$-modules,
\[\lambda_v(f_i)=\op{rank}_{\Z_p} \mathcal{H}_v(f_i/\Q_{\infty})^{\vee}.\]It follows from \eqref{ses} that
\begin{equation}\label{lambdaequation1}\lambda^{\ddag}(f_i)=\lambda^{\ddag, \Sigma}(f_i)-\sum_{v\in \Sigma\backslash \{p\}} \lambda_v(f_i).\end{equation}
The following is a special case of a result due to Hatley and Lei, see \cite[section]{hatleylei} that applies to our setting and generalizes results of Kim \cite{Kim} and Greenberg-Vatsal \cite{greenbergvatsal}. Note that the (BLZ) condition in \cite[section 4.2, p. 1278]{hatleylei} is automatically satisfied since it is assumed that $p\geq k$.
\begin{Th}\label{th35}
Let $f_1$ and $f_2$ be $\mathfrak{p}$-congruent modular forms satisfying Hypothesis $\ref{mainhyp}$ and $\Sigma$ the set of primes dividing $pN_1N_2$. Let $\ddag$ be a choice of sign $+$ or $-$. Then the following assertions hold:
\begin{itemize}
    \item $\mu^{\ddag}(f_1)=0$ if and only if $\mu^{\ddag}(f_2)=0$.
    \item If $\mu^{\ddag}(f_1)=0$ (or equivalently, $\mu^{\ddag}(f_2)=0$) then, 
    \[\lambda^{\ddag, \Sigma}(f_1)=\lambda^{\ddag, \Sigma}(f_2).\]
\end{itemize}
\end{Th}
\begin{Corollary}\label{cor36}
Let $f_1$ and $f_2$ be $\mathfrak{p}$-congruent modular forms satisfying Hypothesis $\ref{mainhyp}$. Let $\ddag\in \{+,-\}$ and assume that $\mu^{\ddag}(f_1)=0$ (or equivalently, $\mu^{\ddag}(f_2)=0$). Suppose that there is a set of $n$ primes $v_1, \dots, v_n$ not equal to $p$ such that
\begin{enumerate}
    \item $\bar{\rho}_{f_1}(g)=\op{Id}$ (or equivalently $\bar{\rho}_{f_2}(g)=\op{Id}$) for all $g\in \op{G}_v$.
    \item For $i=1,\dots, n$, the prime $v_i$ divides $N_2$ but not $N_1$.
\end{enumerate}
Then, we have that
\[\lambda^{\ddag}(f_2)-\lambda^{\ddag}(f_1)\geq n.\]
\end{Corollary}
\begin{proof}
From Theorem $\ref{th35}$ and \eqref{lambdaequation1}, it follows that 
\[\lambda^{\ddag}(f_2)-\lambda^{\ddag}(f_1)=\sum_{v\in \Sigma\backslash \{p\}} \left(\lambda_v(f_1)-\lambda_v(f_2)\right).\]
The result follows from Corollary $\ref{cor34}$.
\end{proof}
\section{Lifting Galois Representations}\label{s4}
\par In this section, we discuss known results on lifting global Galois representations which will play a crucial role in proving our results on $\lambda$-invariants. The technique was pioneered by Ramakrishna in \cite{ravi1, ravi2} and further improved by Fakhruddin, Khare and Patrikis in \cite{FKP}. Let $p\geq 5$ be a prime number and $\F$ a finite field of characteristic $p$. Recall that $\chi:\op{G}_{\Q, \{p\}}\rightarrow \op{GL}_1(\Z_p)$ is the cyclotomic character. Depending on the context, we shall sometimes also denote the restriction of $\chi_{|p}$ by $\chi$ itself. Let $c$ denote complex conjugation and $R$ a $\Z_p$-algebra. A character $\alpha:\op{Gal}(\bar{\Q}/\Q)\rightarrow \op{GL}_1(R)$ is \textit{odd} if $\alpha(c)=-1$. We recall the notion of a \textit{geometric} Galois representation, due to Fontaine and Mazur \cite{fontainemazur}. 
\begin{Definition}\label{geometricdef}
Let $\rho:\op{Gal}(\bar{\Q}/\Q)\rightarrow \op{GL}_2(\bar{\Q}_p)$ be a continuous Galois representation. Then, $\rho$ is said to be \textit{geometric} if it satisfies the following conditions:
\begin{enumerate}
    \item $\rho$ is irreducible.
    \item $\rho$ is unramified away from finitely many primes.
    \item The determinant character $\det \rho$ is odd.
    \item The local representation $\rho_{\restriction p}:\op{G}_p\rightarrow \op{GL}_2(\bar{\Q}_p)$ is \textit{deRham} in the sense of \cite[p.73]{brinonconrad}. Note that this condition is equivalent to requiring that $\rho_{|p}$ be \emph{potentially semistable} (cf. page 78, line 31, and Theorem 9.3.5 of \emph{loc. cit.}).
\end{enumerate}
\end{Definition} If $\mathcal{O}$ is the valuation ring of a finite extension $K$ of $\Q_p$, a continuous integral representation $\rho:\op{Gal}(\bar{\Q}/\Q)\rightarrow \op{GL}_2(\mathcal{O})$ is said to be geometric if it satisfies the conditions above, when viewed as a representation to $\op{GL}_2(K)$. The Fontaine-Mazur conjecture states that a geometric Galois representation arises from a Hecke eigencuspform. More generally, an $n$-dimensional geometric Galois representation is expected to arise from the \'etale cohomology of a proper variety, hence the terminology. \par Denote by $\op{W}(\F)$ the ring of Witt vectors with residue field $\F$. This is the valuation ring of the unique unramified extension of $\Q_p$ with residue field $\F$. Let $\bar{\rho}:\op{Gal}(\bar{\Q}/\Q)\rightarrow \op{GL}_2(\F)$ be a Galois representation. Serre conjectured that if $\bar{\rho}$ is irreducible and $\op{det}\bar{\rho}$ is an odd character, then it necessarily lifts to a characteristic zero Galois representation arising from a Hecke eigencuspform. Ramakrishna proved in \cite{ravi2} that if $\bar{\rho}$ satisfies some additional conditions, then it lifts to a characteristic zero geometric Galois representation $\rho$ as depicted
\[ \begin{tikzpicture}[node distance = 2.5 cm, auto]
            \node at (0,0) (G) {$\op{Gal}(\bar{\Q}/\Q)$};
             \node (A) at (3,0){$\op{GL}_2(\F)$.};
             \node (B) at (3,2){$\op{GL}_2(\op{W}(\F))$};
      \draw[->] (G) to node [swap]{$\bar{\rho}$} (A);
       \draw[->] (B) to node{} (A);
      \draw[->] (G) to node {$\rho$} (B);
\end{tikzpicture}\]Subsequently, Khare and Wintenberger in \cite{kharewintenberger} proved Serre's conjecture and all but a few cases of the Fontaine-Mazur conjecture would be resolved, by the work of Taylor \cite{taylor}, Kisin \cite{kisin} and others.

\par We introduce some notions from the deformation theory of Galois representations. Let $p\geq 5$ be a prime and $S$ be a finite set of primes containing $p$. Fix a mod-$p$ Galois representation
\[\bar{\rho}:\op{G}_{\Q,S}\rightarrow\op{GL}_2(\F)\] and at each prime $v\in S$, denote by $\bar{\rho}_{|v}$ the restriction of $\bar{\rho}$ to $\op{G}_v$.\begin{Definition} Let $\mathcal{C}_{\op{W}(\F)}$ be the category of coefficient rings over $\text{W}(\F)$ with residue field $\F$. The objects of this category consist of local $\text{W}(\F)$-algebras $(R,\mathfrak{m})$ for which
      \begin{itemize}
          \item $R$ is complete and noetherian,
          \item $R/\mathfrak{m}$ is isomorphic to $\F$ as a $\text{W}(\F)$-algebra.
      \end{itemize}A morphism in this category is a map of local rings which is also a $\text{W}(\F)$-algebra homorphism. \end{Definition}
      A coefficient ring $R$ is equipped with a residue map $R\rightarrow \F$, upon going modulo the maximal ideal.  For $(R, \mathfrak{m})\in \mathcal{C}_{\op{W}(\F)}$, let $\widehat{\op{GL}_2}(R)\subset \op{GL}_2(R)$ be the subgroup of matrices which reduce to the identity modulo the maximal ideal.
      \begin{Definition} Let $R$ denote a coefficient ring with residue field $\F$ and $\Pi$ denote $\op{G}_{\Q}$ (resp. $\op{G}_{v}$). Two lifts $\rho,\rho':\Pi\rightarrow \op{GL}_2(R)$ of $\bar{\rho}$ (resp. $\bar{\rho}_{\restriction v}$) the residual representation are \textit{strictly-equivalent} if $\rho=A\rho' A^{-1}$ for $A\in \widehat{\op{GL}}_2(R)$. An \textit{$R$-deformation} is a strict equivalence class of lifts.
      \end{Definition}
    Note that if $\rho$ and $\rho'$ are strictly equivalent, then $\op{det} \rho=\op{det}\rho'$. If $\rho$ is a deformation, the determinant $\op{det}\rho$ is thus a well defined character independent of the choice of representative. Let $\psi:\op{G}_{\Q, S}\rightarrow \op{GL}_1(\op{W}(\F))$ be a lift of $\det\bar{\rho}$. For any $\op{W}(\F)$-algebra $R$, we may view $\psi$ as a homomorphism into $\op{GL}_1(R)$. At each prime $v$, set $\operatorname{Def}_v^{\psi}(R)$ to be the set of $R$-deformations of $\bar{\rho}_{\restriction v}$ with determinant equal to $\psi_{|v}:\op{G}_{v}\rightarrow \op{GL}_1(R)$. The association $R\mapsto \operatorname{Def}_v^{\psi}(R)$ defines a functor $\operatorname{Def}_v^{\psi}:\mathcal{C}_{\op{W}(\F)}\rightarrow \operatorname{Sets}$. A deformation functor at $v$ with determinant $\psi_{|v}$ is a subfunctor $\mathcal{C}_v$ of $\op{Def}_v^{\psi}$. For each coefficient ring $R$, the set $\mathcal{C}_v(R)$ is a subset of $\op{Def}_v^{\psi}(R)$.
    \begin{Definition}\label{adgaloisaction}Set $\op{Ad}\bar{\rho}$ to denote the Galois module whose underlying vector space consists of $2\times 2$ matrices with entries in $\F$. The Galois action is as follows: for $g\in \op{G}_{\Q}$ and $v\in \op{Ad}\bar{\rho}$, 
     set $g\cdot v:=\bar{\rho}(g) v \bar{\rho}(g)^{-1}$. Let $\g$ be the $\op{G}_{\Q}$-stable submodule of trace zero matrices and $\g^*:=\op{Hom}_{\F}(\g, \F(\bar{\chi}))$.
     \end{Definition}

Let $\widetilde{\eta}:\op{G}_v\rightarrow \op{Ad}^0\bar{\rho}$ be a $1$-cocyle and let $\eta\in H^1(\op{G}_v, \op{Ad}^0\bar{\rho})$ be the associated cohomology class. Set $D:=\F[\epsilon]=\F[x]/(x^2)$ be the ring of dual numbers, with $\epsilon:=x(\mod{x^2})$. Let $\op{Id}$ denote the $2\times 2$ identity matrix. Then, $\bar{\rho}_{|v}(\op{Id}+\epsilon \widetilde{\eta}(\cdot))$ defines a lift of $\bar{\rho}_{|v}$. The deformation class of this lift is denoted by $\bar{\rho}_{|v}(\op{Id}+\epsilon \eta(\cdot))$, and depends only on the cohomology class $\eta$. It is easy to see that this gives an identification of $\operatorname{Def}_v^{\psi}(D)$ with $H^1(\op{G}_v, \g)$ (for any choice of character $\psi:\op{G}_v\rightarrow \op{GL}_1(\op{W}(\F))$). Given a liftable deformation function $\mathcal{C}_v$, the set of infinitesimal deformations $\mathcal{C}_v(D)$ is an $\F_p$-subspace of $H^1(\op{G}_v, \g)$.
     
    \begin{Definition}
    Let $R\in \mathcal{C}_{\op{W}(\F)}$ with maximal ideal $\mathfrak{m}$, and $I$ an ideal in $R$. The mod-$I$ reduction map $R\rightarrow R/I$ is said to be a \textit{small extension} if $I.\mathfrak{m}=0$. A functor of deformations $\mathcal{C}_v:\mathcal{C}_{\op{W}(\F)}\rightarrow \op{Sets}$ of $\bar{\rho}_{\restriction v}$ is \textit{liftable} if for every small extension $R\rightarrow R/I$ the induced map $\mathcal{C}_v(R)\rightarrow \mathcal{C}_v(R/I)$ is surjective.
    \end{Definition}

    \begin{Proposition}\label{prop46}
    Let $p\geq 5$ be a prime number, and $\bar{\rho}:\op{Gal}(\bar{\Q}/\Q)\rightarrow \op{GL}_2(\F)$ be a Galois representation. Assume that the determinant character $\det \bar{\rho}$ is odd and that $\bar{\rho}_{|p}$ is not twist equivalent to the trivial representation or an unramified indecomposable representation of the form $\mtx{1}{\ast}{0}{1}$. At a prime $v\in S$, set \[\delta_{v}:=\begin{cases}
     2\text{ if }v=p,\\
    0 \text{ otherwise}.\end{cases}\] Then, at each prime $v\in S$ there is a liftable deformation functor $\mathcal{C}_v=\mathcal{C}_v^\psi\subseteq \op{Def}_v^{\psi}$ such that $\mathcal{C}_v(D)$ is an $\F$-subspace of $H^1(\op{G}_v, \g)$ of codimension $H^2(\op{G}_v, \g)+\delta_v$. Furthermore, at $p$, the functor $\mathcal{C}_p\subseteq \op{Def}_p^{\psi}$ satisfies the following properties.
    \begin{enumerate}
        \item Suppose that $\bar{\rho}_{|p}$ is reducible of the form $\mtx{\varphi_1}{\ast}{0}{\varphi_2}$, where $\varphi_2$ is an unramified character. Then $\mathcal{C}_p$ consists of ordinary deformations $\mtx{\chi^a \gamma_1}{\ast}{0}{\gamma_2}$. Here, $\gamma_1$ and $\gamma_2$ are characters of finite order and $\gamma_2$ is unramified. In other words, in this case, the deformations are ordinary.
        \item If $\bar{\rho}_{|I_p}$ acts by fundamental characters of level $2$, then the deformations in $\mathcal{C}_p$ are crystalline.
    \end{enumerate}
    \end{Proposition}
    \begin{proof}
    When $v\neq p$, this result follows from \cite[Proposition 1]{ravi2}. When $v=p$, we refer to the tables on p.125 and p.128, and the case by case analysis in pp.124-138 in \textit{loc. cit.} We also refer to the statement of the second assertion of Theorem 1 in \cite{ravi2}. 
    \end{proof}
    Note that for any $m>0$, the quotient map $\op{W}(\F)/p^{m+1}\rightarrow \op{W}(\F)/p^m$ is a small extension. It follows from the above result that $\bar{\rho}_{|v}$ may be lifted to a continuous characteristic zero representation 
    \[r_v: \op{G}_v\rightarrow \op{GL}_2(\op{W}(\F))\] satisfying $\mathcal{C}_v$ with determinant $\psi_{|v}$, by lifting one step at a time, from $\op{W}(\F)/p^m$ to $\op{W}(\F)/p^{m+1}$.
Consider the surjection $\op{GL}_2(\F)\rightarrow \op{PGL}_2(\F)$ obtained by going modulo the subgroup of scalar matrices. The projective image of $\bar{\rho}$ refers to the image of the composed representation to $\op{PGL}_2(\F)$. 
\par The result below follows from \cite[Theorem A]{FKP}, when specialized to the group $\op{G}=\op{GL}_2$ and the number field $F=\Q$. This result in indeed very general and applies to a general algebraic group, hence the assumptions are rather technical. The author refers to this particular result since it is the only one in the literature that he is aware of that makes an additional assertion about fixing the congruence class of local representations $r_v$ at a set of primes $\Omega$.
\begin{Th}\label{techtheorem}
Let $p\geq 5$ be a prime number and $\F$ a finite field of characteristic $p$. Let $S$ be a finite set of primes containing $p$ and \[\bar{\rho}:\op{G}_{\Q,S}\rightarrow \op{GL}_2(\F)\] a Galois representation with odd determinant. Let $\psi:\op{G}_{\Q,S}\rightarrow \op{GL}_1(\op{W}(\F))$ be an odd continuous lift of $\det \bar{\rho}$ and $\Omega$ be a subset of $S\backslash \{p\}$. Assume that the following conditions are satisfied:

\begin{enumerate}
    \item the projective image of $\bar{\rho}$ contains $\op{PSL}_2(\F_p)$.
    \item For each prime $v\in \Omega$, $\bar{\rho}_{|v}$ is prescribed a lift 
    \[r_v:\op{G}_v\rightarrow \op{GL}_2(\op{W}(\F))\] with determinant $\psi_{|v}$.
    \item If $\bar{\rho}_{|p}$ is reducible, then $\bar{\rho}_{|p}$ is not twist equivalent to the trivial representation or an unramified indecomposable representation of the form $\mtx{1}{\ast}{0}{1}$.
\end{enumerate}Let $N\geq 1$ be any positive integer. Then, there is a finite set of primes $\tilde{S}\supseteq S$ and a continuous lift
\[\rho:\op{G}_{\Q, \tilde{S}}\rightarrow \op{GL}_2(\mathcal{O})\] of $\bar{\rho}$ such that the following conditions are satisfied.
\begin{enumerate}
    \item For each prime $v\in \Omega$, 
    \[\rho_{|v}\equiv r_v \mod{\varpi^N},\]
    \item $\rho$ is geometric in the sense of Definition $\ref{geometricdef}$,
    \item $\rho_{|p}$ satisfies $\mathcal{C}_p$ (see Proposition $\ref{prop46}$),
    \item $\op{det}\rho=\psi$.
\end{enumerate}Here, $\mathcal{O}$ is the valuation ring of a finite extension of $\op{W}(\F)[p^{-1}]$ and $\varpi$ its uniformizer. The choice of $\mathcal{O}$ depends on $\{r_v\}_{v\in \Omega}$ and not on $N$ (but the lift $\rho$ does depend on $N$).
\end{Th}
\begin{proof}
As mentioned earlier, the result follows from \cite[Theorem A]{FKP}, though it takes some explanation to show how various technical hypotheses specialize to this simpler setting. We refer to notation from \textit{loc. cit.} in order to clarify the hypotheses.
\begin{itemize}
    \item First, it should be noted that since the aforementioned result is far more general than the case considered, it is only stated that the result applies for $p$ suitably large, depending on the group. As stated in Remark 6.17 of \textit{loc. cit.}, the assumption on the prime $p$ is in place in order to ensure disjointedness of $\g$ and $\g^*$ as Galois modules. However, since we are working with $\op{G}=\op{GL}_2$ and have made the simplifying assumption that the projective image of $\bar{\rho}$ contains $\op{PSL}_2(\F_p)$, it is an easy exercise to show that the Galois modules $\g$ and $\g^*$ are irreducible and non-isomorphic. Hence, this requirement on the prime $p$ is automatic in our simplified setting and applies for $p\geq 5$. It is for the same reason that the assumption $[\tilde{\Q}(\mu_p):\tilde{\Q}]\geq a_G$ is in place, and this assumption may be dropped when the modules $\g$ and $\g^*$ are irreducible and non-isomorphic as Galois modules.
    \item Since $\op{PSL}_2(\F_p)$ is not solvable, it follows that $\op{\rho}_{|\op{G}_{\Q(\mu_p)}}$ is absolutely irreducible.
    \item The results Proposition 6.8 and Theorem 6.9 replace Ramakrishna's original Selmer vanishing argument for $\op{GL}_2$. These assumptions are thus no longer necessary when $\op{G}=\op{GL}_2$, since the original result of Ramakrishna (see \cite[Lemma 16]{ravi2}) applies. This is indeed stated in the last two lines of \cite[p. 46]{FKP}.
    \item It follows from Proposition \ref{prop46} that at each place $v\in S$, we may choose a lift $r_v$ of $\bar{\rho}_{|v}$ satisfying $\mathcal{C}_v$ with $\det r_v=\psi_{|v}$.
    \item The local deformation $\rho_{|p}$ is arranged to satisfy the local condition $\mathcal{C}_p$. This goes back to the original construction of Ramakrishna \cite{ravi1, ravi2}. In the more general setting where $\op{G}$ is a general group, there is no suitable analog for the liftable condition $\mathcal{C}_p$, which is why instead in \textit{loc. cit.}, the local deformation $\rho_{|p}$ is arranged to be in a chosen component of the generic fiber of a local lifting ring.
\end{itemize}   Therefore, the results in \cite{FKP} specialize to the case considered in the Theorem.
\end{proof}
\section{Modular deformations with large $\lambda$-invariant}\label{s5}
In this section, we prove the main results of the paper. Recall that $\chi$ denotes the cyclotomic character and set $\bar{\chi}$ for its mod-$p$ reduction. When there is no cause for confusion, $\bar{\chi}_{|p}$ will simply be denoted $\bar{\chi}$. The following result is used to prove Theorem \ref{th52}.

\begin{Th}\label{th51}
Let $p\geq 5$ and $\F$ a finite field of characteristic $p$. Let $f_1$ be a Hecke eigencuspform of weight $k\geq 2$ on $\Gamma_1(N_1)$ and let $\bar{\rho}:=\bar{\rho}_{f_1}:\op{Gal}(\bar{\Q}/\Q)\rightarrow \op{GL}_2(\F)$ be the residual representation. Assume that the following conditions are satisfied:
\begin{enumerate}
    \item $f_1$ is either $p$-ordinary or $p\nmid N_1$ and $p\geq k$ (also referred to as condition $(\star)$).
    \item The projective image of $\bar{\rho}$ contains $\op{PSL}_2(\F_p)$.
    \item The local representation $\bar{\rho}_{|p}$ is not twist equivalent to the trivial representation or an unramified indecomposable representation of the form $\mtx{1}{\ast}{0}{1}$.
\end{enumerate} Let $n>0$ be a positive integer. Then, $\bar{\rho}$ lifts to a representation 
\[\rho_{f_2}:\op{Gal}(\bar{\Q}/\Q)\rightarrow \op{GL}_2(\mathcal{O})\] such that $f_2$ is a Hecke eigencuspform on $\Gamma_1(N_2)$ of weight $k$. Here, $\mathcal{O}$ is the valuation ring of a finite extension of $\op{W}(\F)[p^{-1}]$. Furthermore, the following conditions are satisfied:
\begin{enumerate}
    \item $f_2$ is either $p$-ordinary or $p\nmid N_2$,
    \item there is a set of primes $v_1,\dots, v_n$ such that $v_i\equiv 1 \mod{p}$ and $\bar{\rho}_{|v_i}$ is trivial and the primes $v_1,\dots, v_n$ divide $N_2$ but not $N_1$. 
\end{enumerate}
\end{Th}
\begin{proof}
Denote by $\Q(\bar{\rho})$ the field fixed by the kernel of $\bar{\rho}$ and let $L$ be the composite $L:=\Q(\bar{\rho})\cdot \Q(\mu_p)$. Let $\Pi$ be the set of prime numbers $v$ that split in $L$. By the Chebotarev density theorem, the density of $\Pi$ is equal to $[L:\Q]^{-1}$. In particular, the set $\Pi$ is infinite. Choose primes $v_1, \dots, v_n\in \Pi$ not dividing $N_1$. Note that $\bar{\rho}$ is unramified at $v_i$ and $\bar{\rho}_{\restriction v_i}$ is the trivial representation since $v_i$ splits in $\Q(\bar{\rho})$. Since $v_i$ splits in $\Q(\mu_p)$, we have that $v_i\equiv 1\mod{p}$. At each prime $v_i$, we will specify a lift \[r_{v_i}:\op{G}_{v_i}\rightarrow \op{GL}_2(\op{W}(\F))\] of $\bar{\rho}_{|v_i}$ which is ramified modulo $p^2$ and $\op{det}r_{v_i}=\psi_{|v_i}$. Since $\bar{\rho}_{\restriction v_i}$ is trivial, the extension of $\Q_{v_i}$ fixed by the kernel of $r_{v_i}$ is pro-$p$ and hence tamely ramified. In other words, $r_{v_i}$ must factor through the maximal tamely ramified pro-$p$ quotient of $\op{G}_{v_i}$. This quotient is a semi-direct product generated by two elements $\sigma$, the Frobenius, and $\tau$, a choice of generator for tame pro-$p$ inertia. These elements are subject to a single relation $\sigma \tau \sigma^{-1}=\tau^{v_i}$. Specifying a lift $r_{v_i}$ amounts to specifying matrices $r_{v_i}(\sigma), r_{v_i}(\tau)\in \widehat{\op{GL}_2}(\op{W}(\F))$ such that 
\[r_{v_i}(\sigma) r_{v_i}(\tau) r_{v_i}(\sigma)^{-1}=r_{v_i}(\tau)^{v_i}.\] Let $y\in \op{W}(\F)$ be such that $p| y$ and $p^2\nmid y$. Pick a square root $v_i^{1/2}$ of $v_i$ in $\Z_p$ which is $1\mod{p}$. Note that since $\bar{\rho}(\sigma)=\mtx{1}{0}{0}{1}$, we have that $\bar{\psi}(\sigma)=\op{det}\bar{\rho}(\sigma)=1$. Set 
\[r_{v_i}(\sigma)=(\psi(\sigma)v_i^{-1})^{\frac{1}{2}}\mtx{v_i}{0}{0}{1}\text{, and }r_{v_i}(\tau)=\mtx{1}{y}{0}{1}.\] Since $v_i\equiv 1\mod{p}$, the matrix $(\psi(\sigma)v_i^{-1})^{\frac{1}{2}}\mtx{v_i}{0}{0}{1}$ lies in $ \widehat{\op{GL}_2}(\op{W}(\F))$. The relation is satisfied and this specifies $r_{v_i}$. Let $\Omega$ be the set of primes $\{v_1, \dots, v_n\}\subset \Pi$ chosen. Let $S$ be the set consisting of the following primes $v$
\begin{enumerate}
    \item $v=p$,
    \item $\bar{\rho}_{|v}$ is ramified,
    \item $v\in \Omega$.
\end{enumerate} The conditions of Theorem \ref{techtheorem} are satisfied and hence there is a lift
\[\rho:\op{Gal}(\bar{\Q}/\Q)\rightarrow \op{GL}_2(\mathcal{O})\] with determinant $\psi$ and which is geometric in the sense of Definition \ref{geometricdef}. Here, $\mathcal{O}$ is an integral extension of $\op{W}(\F)$ which depends only on $\{r_v\}_{v\in \Omega}$. Pick $N$ large enough such that $p\notin \varpi^N$. For this choice of $N$, the representation $\rho$ (which depends on $N$) satisfies the congruence \[\rho_{|v}\equiv r_{v}\mod{\varpi^N}\] for all $v\in \Omega$. For $v\in \Omega$, $r_{v}$ is chosen to be ramified modulo $p^2$. Since $p\notin \varpi^N$, it follows that $r_{v}\mod{\varpi^N}$ is ramified for $v\in \Omega$. Hence, $\rho_{|v}$ is ramified for $v\in \Omega$.
\par Note that since $f_1$ has weight $k$, the character $\psi$ is the product of a finite order character with $\chi^{k-1}$. The determinant of $\rho$ is equal to $\psi$. The representation $\rho_{|p}$ satisfies $\mathcal{C}_p$, hence is ordinary or crystalline (or both). Since $\bar{\rho}_{|p}$ is stipulated to not be twist equivalent to $\mtx{\bar{\chi}}{\ast}{0}{1}$, the assumptions of the main result of \cite{kisin} are satisfied and hence, $\rho$ arises from a Hecke eigencuspform $f_2$ on $\Gamma_1(N_2)$ of weight $k$. Since $\rho$ is $p$-ordinary or $p$-crystalline (or both), it follows that $f_2$ is either $p$-ordinary or $p\nmid N_2$. Since $\rho$ is ramified at the primes $v_1,\dots , v_n$, it follows that $v_1,\dots, v_n$ all divide $N_2$. Also recall that the primes $v_1,\dots , v_n$ were chosen to not divide $N_1$.
\end{proof}

\begin{Th}\label{th52}
Let $p\geq 5$ and $f_1$ a Hecke eigencuspform of weight $k\geq 2$ on $\Gamma_1(N_1)$ and let $\bar{\rho}:=\bar{\rho}_{f_1}:\op{Gal}(\bar{\Q}/\Q)\rightarrow \op{GL}_2(\F)$ be the residual representation. Assume that the following conditions are satisfied:
\begin{enumerate}
    \item if $f_1$ is not $p$-ordinary, then $p$ does not divide $N_1$ and $p\geq k$,
    \item the projective image of $\bar{\rho}$ contains $\op{PSL}_2(\F_p)$,
    \item $\bar{\rho}_{|p}$ is not twist equivalent to the trivial representation or an unramified indecomposable representation of the form $\mtx{1}{\ast}{0}{1}$.
    \item If $k=2$ and $f_1$ is $p$-supersingular with $a_p(f_1)\neq 0$, then, assume that Hypothesis \ref{hypothesis torsion} is satisfied for $f_1$.
    \item $\mu_p(f_1)=0$ (resp. $\mu_p^{+}(f_1)=0,\mu_p^{-}(f_1)=0$) if $f_1$ is (resp. is not) $p$-ordinary. 
\end{enumerate} Let $n>0$ be a positive integer. Then, there is another Hecke eigencuspform $f_2$ of weight $k$ on $\Gamma_1(N_2)$ such that
\begin{enumerate}
    \item $f_2$ is $\mathfrak{p}$-congruent to $f_1$ for some prime $\mathfrak{p}|p$,
    \item $f_2$ is either $p$-ordinary or $p\nmid N_2$.
    \item If $k=2$ and $f_2$ is $p$-supersingular with $a_p(f_2)\neq 0$, then, assume that Hypothesis \ref{hypothesis torsion} is satisfied for $f_2$. Then, with this additional condition in place, we have that $\lambda_p(f_2)\geq n$ (resp. $\lambda_p^{+}(f_2),\lambda_p^{-}(f_2)\geq n$) if $f_2$ is (resp. is not) $p$-ordinary. 
\end{enumerate}
\end{Th}
\begin{Remark}
This result is shown to follow from Corollary \ref{cor36} and Theorem \ref{th51}. Note that Convention \ref{convention} is used in the statement of Corollary \ref{cor36} but not in the above Theorem.
\end{Remark}
\begin{proof}
The eigencuspform $f_1$ satisfies the conditions of Theorem \ref{th51} and hence there is an eigencuspform $f_2$ such that $\rho_{f_2}$ lifts the residual representation $\bar{\rho}$. As a result the eigencuspforms $f_1$ and $f_2$ are $\mathfrak{p}$ congruent for a certain choice of prime $\mathfrak{p}|p$. It follows from Corollary \ref{cor36} that $\lambda_p(f_2)\geq n$ (resp. $\lambda_p^{+}(f_2),\lambda_p^{-}(f_2)\geq n$) if $f$ is (resp. is not) $p$-ordinary. 
\end{proof}

The next Theorem gives an example for each prime $p\geq 7$ and $p\nmid N$, thus it gives an infinite collection of examples which arise from a single modular form.
\begin{Th}\label{lastresult}
Let $p\geq 7$ be any prime and $n$ any positive integer. Let $f_1$ be a non-CM newform of weight $3$, level $M=27$ and and nebentypus $\epsilon(a)=\left(\frac{-3}{a}\right)$. Assume that $\mu_p(f_1)=0$ (resp. $\mu_p^{+}(f_1)=0,\mu_p^{-}(f_1)=0$) if $f_1$ is (resp. is not) $p$-ordinary. There is a Hecke eigencuspform $f$ of weight $3$ on $\Gamma_1(N)$ such that
\begin{enumerate}
    \item if $f$ is not $p$-ordinary, then $p\nmid N$.
    \item If $f_1$ is $p$-ordinary, then so is $f$ and $\lambda_p(f)\geq n$.
    \item If $f$ is not $p$-ordinary, then $f$ is not $p$-ordinary and $p\nmid N$. Furthermore, we have that $\lambda_p^+(f), \lambda_p^-(f)\geq n$.
\end{enumerate} The choice of $f$ depends on $p$ and $n$.
\end{Th}
\begin{proof}
Let $\bar{\rho}$ be the residual representation of $\rho_{f_1}$, for a certain choice of prime $\mathfrak{p}|p$. It follows from \cite[section 1.2]{zywina} that the projective image of $\bar{\rho}$ contains $\op{PSL}_2(\F_p)$. Furthermore, since $k=3$ and, it follows that $\bar{\rho}$ is not twist equivalent to $\mtx{1}{\ast}{0}{1}$. This is because $\bar{\rho}$ is either irreducible or $\bar{\rho}_{|I_p}=\mtx{\bar{\chi}^2}{\ast}{0}{1}$. Note that $p\nmid M$ since $p\neq 3$. According to Theorem \ref{th52}, there is an eigencuspform $f_2$ which is $\mathfrak{p}$-congruent to $f_1$ such that the $\lambda$-invariant is $\geq n$. Since $f_1$ and $f_2$ are $\mathfrak{p}$-congruent, it follows from Lemma \ref{congruencelemma} that $f_1$ is $p$-ordinary if and only if $f_2$ is $p$-ordinary as well. Set $f$ to denote $f_2$.
\end{proof}

\section*{Funding}
The author is supported by the CRM-Simons fellowship.

\section*{Acknowledgments}
The author thanks Jeffrey Hatley, Antonio Lei and Florian Sprung for helpful comments with regards to an earlier draft of the paper. He would also like to extend his sincerest thanks to the anonymous referee who has done an excellent and thorough job in reviewing the manuscript.

\end{document}